\documentclass{article}
\usepackage{graphicx} % Required for inserting images
\usepackage{amsfonts}
\usepackage{amscd,color}
\usepackage{amsmath,amsfonts,amssymb,amscd}
\usepackage{booktabs}
\usepackage{mathrsfs}
\usepackage{indentfirst,graphicx,epsfig}
\usepackage{xcolor}
\usepackage{epstopdf}
\usepackage{caption}
\usepackage{ifpdf}
\usepackage{float}
\usepackage{footnote}
\usepackage{epsfig}
\usepackage{latexsym}
\usepackage{caption}
\usepackage{amsmath,amsthm}
\usepackage{amsfonts}
\usepackage{amssymb}
\usepackage{graphicx}
\usepackage{graphics}
\usepackage{bm}
\usepackage{color}
\usepackage{multirow}
\usepackage{subfigure}
\usepackage{threeparttable}
\usepackage{ulem}

\makeatletter

\newcommand{\Rmnum}[1]{\expandafter\@slowromancap\romannumeral #1@}
\makeatother
\makeatletter
\let\@fnsymbol\@arabic
\makeatother

\makeatletter
  \newcommand\figcaption{\def\@captype{figure}\caption}
  \newcommand\tabcaption{\def\@captype{table}\caption}
\makeatother
\begin{document}
\newtheorem{theorem}{Theorem}[section]
\newtheorem{observation}[theorem]{Observation}
\newtheorem{corollary}[theorem]{Corollary}
\newtheorem{algorithm}[theorem]{Algorithm}
\newtheorem{problem}[theorem]{Problem}
\newtheorem{question}[theorem]{Question}
\newtheorem{lemma}[theorem]{Lemma}
\newtheorem{proposition}[theorem]{Proposition}
\newtheorem{definition}[theorem]{Definition}
\newtheorem{guess}[theorem]{Conjecture}
\newtheorem{claim}[theorem]{Claim}
\newtheorem{example}[theorem]{Example}
\newtheorem{acknowledgement}[theorem]{Acknowledgement}
\newtheorem{axiom}[theorem]{Axiom}
\newtheorem{case}[theorem]{Case}
\newtheorem{conclusion}[theorem]{Conclusion}
\newtheorem{conjecture}[theorem]{Conjecture}
\newtheorem{criterion}[theorem]{Criterion}
\newtheorem{exercise}[theorem]{Exercise}
\newtheorem{notation}[theorem]{Notation}
\newtheorem{solution}[theorem]{Solution}
\newtheorem{summary}[theorem]{Summary}
\newtheorem{fact}[theorem]{Fact}
\newtheorem{remark}[theorem]{Remark}
\newtheorem{subclaim}[theorem]{Subclaim}

\newcommand{\pp}{{\it p.}}
\newcommand{\de}{\em}
\newcommand{\mad}{\rm mad}

\newcommand*{\QEDA}{\hfill\ensuremath{\blacksquare}}
\newcommand*{\QEDB}{\hfill\ensuremath{\square}}
\newcommand{\xz}[1]{{\color{purple}[XZ: #1]}}
\newcommand{\zw}[1]{{\color{blue} [zw: #1]}}
\newcommand{\qf}{Q({\cal F},s)}
\newcommand{\qff}{Q({\cal F}',s)}
\newcommand{\qfff}{Q({\cal F}'',s)}
\newcommand{\f}{{\cal F}}
\newcommand{\ff}{{\cal F}'}
\newcommand{\fff}{{\cal F}''}
\newcommand{\fs}{{\cal F},s}
\newcommand{\g}{\gamma}
\newcommand{\wrt}{with respect to }

\title{Finiteness of non-decomposable critically 4 and 5-frustrated signed graphs}
\author{\normalsize Zhiqian Wang$^{1}$\\
\small $^{1}$School of Mathematical Sciences, Zhejiang Normal University, Jinhua, China\\
\small Email: 1522686578@qq.com\\
}

\date{}
\maketitle

\begin{abstract}
A signed graph $(G,\sigma)$ is a graph $G$ with a signature $\sigma$ labeling each edge with a positive or negative sign. Two signatures of $G$ are switching equivalent if one is obtained from the other by changing the signs of all edges in an edge-cut. 
The frustration index of a signed graph $(G, \sigma)$ is  the minimum number of negative edges among all  signatures  equivalent to $\sigma$.
A signed graph is critically $k$-frustrated if it has frustration index $k$, and the removal of any edge decreases its frustration index. A critically $k$-frustrated signed graph is prime if it has no subdivided edge (including multiedge) and none of its  subgraphs is the edge-disjoint union of critically frustrated signed graphs. Steffen and Naserasr et al. conjectured that for any positive integer $k$, there are finitely many prime critically $k$-frustrated signed graphs. The cases $k=1,2,3$ have been proved to be true recently by Cappello et al.. In this paper, we show that the conjecture holds when $k=4$ and $5$.
\end{abstract}
\noindent \textbf{Keywords:}
signed graphs; frustration index; projective plane

\section[Introduction]{Introduction}
 
Graphs in this paper may have multiedges and loops. A \textit{signed graph} $(G,\sigma)$ is a graph $G$ together with a \textit{signature} $\sigma$, which is a mapping $\sigma$: $E(G)\mapsto \{+,-\}$  that assigns to each edge a positive sign or a negative sign. An edge $e$ of $(G,\sigma)$ is called \textit{negative} if $\sigma(e)=-$ and \textit{positive} otherwise. The set of positive and negative edges are denoted by $E_\sigma^+(G)$ and $E_\sigma^-(G)$, respectively. When the underlying graph $G$ and the signature $\sigma$ is clear from the context, we may use $E^+$ and $E^-$  instead. We denote by $(G,-)$ (or $(G,+)$) the signed graph with all edges negative (or positive).  

For $X\subseteq V(G)$, let $X^c:=V(G)\setminus X$. Let $\partial_G(X):=\{xy\in E(G): x\in X, y\in X^c\}$, which is called an \textit{edge cut} of $G$. The cardinality of $\partial_G(X)$ is denoted by $d_G(X)$. 
 For a signed graph $(G, \sigma)$, let 
\begin{displaymath}
    d_{(G,\sigma)}^+(X)=|\partial_G(X)\cap E_\sigma^+(G)| \mbox{ and }  d_{(G,\sigma)}^-(X)=|\partial_G(X)\cap E_\sigma^-(G)|.
\end{displaymath}

 Whenever it is clear from the context, we may omit the subscript $(G,\sigma)$.  An edge cut $\partial_G(X)$ is called an $(a,b)$-cut if $d_{(G,\sigma)}^+(X)=a$ and $d_{(G,\sigma)}^-(X)=b$. An $(a,a)$-cut  is called an \textit{equilibrated cut}.

A \textit{cycle} in $G$ is a connected 2-regular subgraph. A cycle is   \textit{positive} (or \textit{negative}) if it contains an even number (or odd number) of negative edges. A signed graph $(G, \sigma)$ is called \textit{balanced} if it does not contain negative cycle and \textit{unbalanced} otherwise.

 For an edge cut $\partial_G(X)$, \textit{switching} at $\partial_G(X)$ means  changing the signs of all the edges in $\partial_G(X)$, and the signs of all the other edges remain unchanged. Two signatures $\sigma$ and $\pi$  are  \textit{switching equivalent} if one is obtained from the other by a switching  at some edge cut. We say that $(G,\sigma)$ and $(G,\pi)$ are switching equivalent when $\sigma$ and $\pi$ are  switching equivalent signatures. It was shown in \cite{82} that two signatures on the same graph are switching equivalent if and only if they have the same set of negative cycles.

The \textit{frustration index} of a signed graph $(G,\sigma)$, denoted by $l(G,\sigma)$, is defined as 
\begin{displaymath}
l(G,\sigma)=\min\{|E_\pi^-(G)|:(G,\pi) \mbox{ is switching equivalent to } (G,\sigma)\}.
\end{displaymath}
If $l(G,\sigma)=k$, then $(G,\sigma)$ is said to be \textit{$k$-frustrated}. A signature $\sigma$ is said to be a \textit{minimum signature} if $|E_\sigma^-(G)|=l(G,\sigma)$.
Thus if $\sigma$ is a minimum signature and $\partial(X)$ is an $(a,b)$-cut of $(G, \sigma)$, then $a \ge b$.

It is known \cite{fi}  that the frustration index $l(G,\sigma)$ of a signed graph equals the least cardinality of a subset $E'$ of $G$ such that $(G-E', \sigma)$ has no negative cycles.  
It follows that for any graph $G$, $l(G,-) = |E(G)|- {\rm mc}(G)$, where  
 ${\rm mc}(G)$ is the number of edges in a   maximum cut of $G$. As it is NP-hard to determine the maximum cut of graphs \cite{np}, it is NP-hard to determine the frustration indices of signed graphs. 
 
The concept of critically $k$-frustrated signed graph was introduced in \cite{22}.
\begin{definition}
A signed graph $(G,\sigma)$ is \textit{critically $k$-frustrated} if $l(G,\sigma)=k$ and for every $e\in E(G)$, $l(G\setminus e,\sigma)=k-1$.
\end{definition}

We may simply say $(G,\sigma)$ is \textit{critical} in short when the exact value of $l(G,\sigma)$ needs not to be emphasized. In \cite{22}, some basic properties of critically $k$-frustrated signed graphs were proved.

\begin{theorem}\label{cut}
(\cite{22}) Let $k$ be a positive integer and $(G,\sigma)$ be a $k$-frustrated signed graph. The following statements are equivalent:
\begin{enumerate}
    \item $(G,\sigma)$ is critically $k$-frustrated.
    \item For every edge $e\in E(G)$, there exists a signature $\sigma^\prime$ equivalent to $\sigma$  such that $|E_{\sigma^\prime}^-(G)|=k$ and $e\in E_{\sigma^\prime}^-(G)$.
    \item If $|E_{\sigma}^-(G)|=k$, then every $e\in E_\sigma^+(G)$ is contained in an equilibrated edge cut of $(G, \sigma)$.
\end{enumerate}
\end{theorem}

A critically $k$-frustrated signed graph $(G,\sigma)$ is said to be \textit{$(k_1,\ldots,k_t)$ -decomposable} if $E(G)$ can be partitioned into $t$ parts $E_1 \cup \ldots \cup E_t$ ($t \ge 2$) such that for $i\in\{1,\ldots,t\}$, the signed subgraph $(G[E_i],\sigma)$ is critically $k_i$-frustrated and $k=k_1+\ldots + k_t$. We simply say $(G,\sigma)$ is \textit{decomposable} when the parameters $k_1,\ldots,k_t$ need not to be emphasized. A critical signed graph is \textit{indecomposable} if it is not decomposable.

Note that if two parallel edges $e_1$ and $e_2$ have different signs, then any equilibrated cut of $(G,\sigma)$ is also an equilibrated cut of $(G\setminus\{e_1,e_2\},\sigma)$. Hence by Theorem \ref{cut}, we have the following observation of indecomposable critical signed graphs.
\begin{observation}
If $(G,\sigma)$ is indecomposable, then it contains no negative loop or two parallel edges of different signs.
\end{observation}

Assume $t \ge 1$ and there is a set $E_{xy}$ of $t$-parallel edges between two vertices $x$ and $y$  of the same sign.  \textit{Subdividing}  $E_{xy}$   means deleting $E_{xy}$ and adding a new vertex $w$, as well as a set $E^+_{xw}$ of $t$ positive  parallel edges between $x$ and $w$  and a set $E_{wy}$ of $t$ parallel edges between $w$ and $y$ that have the same sign as $E_{xy}$. {A signed graph $(G', 
\sigma')$ is a \textit{subdivision} of $(G, \sigma)$ if $(G', 
\sigma')$ is obtained from $(G, \sigma)$ by sequentially subdividing some multiedges, and furthermore, if $(G', 
\sigma')\neq (G, \sigma)$, $(G', 
\sigma')$ is a \textit{proper subdivision} of $(G, \sigma)$.} We say a signed graph is \textit{irreducible} if it is not a proper subdivision of any signed graph.

It was observed in \cite{22}  that if $(G', \sigma')$ is a subdivision of $(G, \sigma)$, then the following claims hold:
  \begin{enumerate}
      \item $l(G', \sigma') = l(G, \sigma)$.
      \item $(G, \sigma)$ is critical if and only if $(G', \sigma')$ is critical.
      \item $(G, \sigma)$ is decomposable if and only if $(G', \sigma')$ is decomposable. 
  \end{enumerate}   

Among all the critical signed graphs, we are interested in the primitive ones performing as elementary structures in other critical signed graphs. They cannot be a subdivision of another signed graph, and moreover, remember that a signed graph is determined by its negative cycles, the set of negative cycles of such signed graphs should not be partitioned. Derived from the multiple weak $2$-linkage problem proposed by Y. Lu et al. \cite{link}, the class of such critical signed graphs were strictly defined by Steffen et al. in \cite{22}.

\begin{problem}
(Multiple Weak $2$-Linkage Problem) For any integer $k\geq 2$, is there a pair of integers $1\leq i<j\leq k$ such that the graph $G$ contains a pair of edge-disjoint paths $P_i$ and $P_j$ such that $P_r$ joins $x_r$ and $y_r$ for each $r=i,j$?
\end{problem}

\begin{definition}\label{def-prime}
A critical signed graph $(G, \sigma)$ is \textit{prime} if it is irreducible and contains no pair of edge-disjoint negative cycles.  
\end{definition}

For a prime critical signed graph $(G, \sigma)$ whence $l(G, \sigma)=k$ and $\sigma$ is a minimum signature, we let $E^-_\sigma(G)=\{x_1y_1,\ldots,x_ky_k\}$ and the all-positive subgraph $G\setminus E^-_\sigma(G)=G^+$. Notice that $(G, \sigma)$ contains no pair of edge-disjoint negative cycles if and only if for any pair of integers $1\leq i<j\leq k$, there does not exist edge-disjoint paths $P_i$ and $P_j$ in $G^+$ joining $\{x_i, y_i\}$ and $\{x_j, y_j\}$ respectively. So the concept prime critical signed graph is virtually the specialization of multiple weak $2$-linkage problem.

Since every negative cycle is a critically $1$-frustrated signed graph, in the perspective of decomposition, it is not difficult to perceive the following equivalent description of prime critical signed graphs.
\begin{observation}
A critical signed graph $(G,\sigma)$ is prime if and only if $(G,\sigma)$ is irreducible and every critical subgraph is indecomposable.
\end{observation}

The following conjecture was proposed in \cite{link}:

\begin{conjecture}
\label{conj-key}
For any positive integer $k$, there are only finitely many prime critically $k$-frustrated signed graphs. 
\end{conjecture}

It was proved in \cite{22} that $C_{-1}$ and $(K_4,-)$ are the only prime critically 1-frustrated and 2-frustrated signed graphs. The case $k=3$ of Conjecture \ref{conj-key} was confirmed in \cite{25}. 

In this paper, we prove that Conjecture \ref{conj-key} holds for cases $k=4$ and $5$. We denote by $\mathcal{S}^*(k)$ the family of prime critically $k$-frustrated signed graphs.

\begin{theorem}\label{mytheorem}
For $k \in \{4,5\}$,  $\mathcal{S}^*(k)$ is finite.
\end{theorem}

\section{Finiteness of $\mathcal{S}^*(4)$ and $\mathcal{S}^*(5)$}
\subsection{$\mathcal{S}^*(k)$ on projective plane} 

The proof of Theorem \ref{mytheorem} is based on a description of signed graphs in $\mathcal{S}^*(k)$, which is derived from the topological characterization of graphs whose answer to multiple weak $2$-linkage problem is false.

\begin{theorem}\label{basis}
(\cite{22}) Let $k\geq 2$ and $(G,\sigma)$ be a signed graph in $\mathcal{S}^*(k)$ with $E^-_\sigma(G)=\{x_1y_1,\ldots, x_ky_k\}$. Then there exists a permutation $\tau$ on $[k]$, and $(G,\sigma)$ is obtained from a 2-connected plane cubic graph $G'$ (all its edges are positive) by selecting a facial circuit $C$, inserting $2k$ distinct vertices $z_1,\ldots, z_{2k}$ on the edges of $C$ in this cyclic order, and adding $k$ negative edges $e_i=x_iy_i$, where $\{z_i, z_{k+i}\}=\{x_{\tau(i)}, y_{\tau(i)}\}$ for $i\in [k]$.

Furthermore, for $k\geq 3$, if $(G,\sigma)\in \mathcal{S}^*(k)$, then the underlying graph $G$ is essentially 4-edge-connected.
\end{theorem}

By Theorem \ref{basis}, a signed graph $(G, \sigma) \in \mathcal{S}^*(k)$ is embedded on the projective plane, with the $k$ negative edges embedded in a cross cap. Such an embedding of $(G, \sigma)$ is called an \textit{canonical embedding}. As an example, it was proved in \cite{25} that $\mathcal{S}^*(3)$  consists of two members, and the canonical embedding of them are shown in Figure \ref{examples}. We use solid and dashed lines in this paper to represent positive and negative edges respectively. The outer dotted cycle in Figure \ref{examples} is used to emphasize the border of the cross cap, and in the remaining figures of the paper, it is omitted for convenience.

\begin{figure}[H]
\centering
\subfigure{\includegraphics[width=4cm]{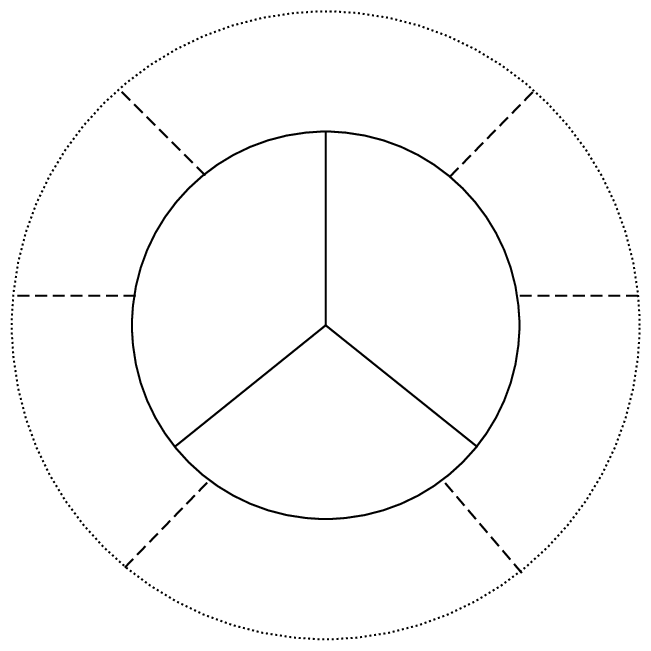}}
\subfigure{\includegraphics[width=4cm]{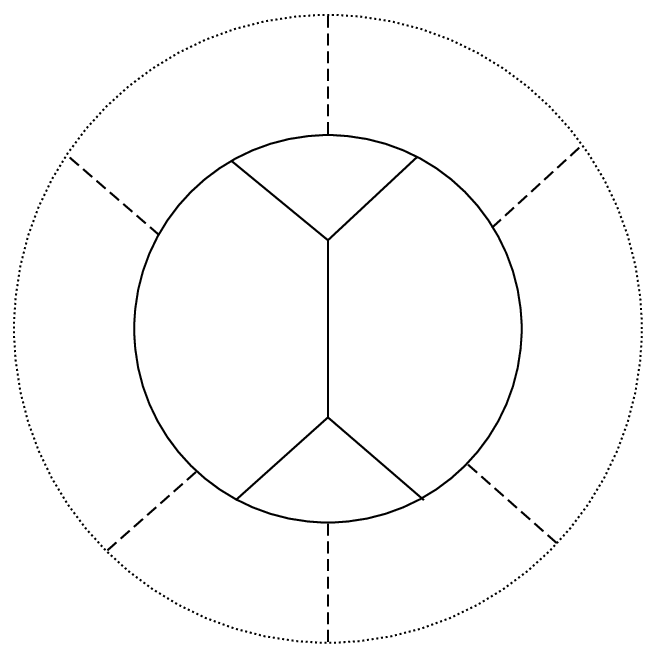}}
\caption{Signed graphs in $\mathcal{S}^*(3)$.}
\label{examples}
\end{figure}
In the following, we assume $(G, \sigma)$ is canonically embedded on the projective plane. The signature $\sigma$ is thus minimum, i.e., the number of negative edges equals the frustration index of $(G, \sigma)$. Note that $G$ is also a cubic graph (with $k$ negative edges).

For $(G,\sigma)\in \mathcal{S}^*(k)$ with $E^-_\sigma(G)=\{x_1y_1,\ldots, x_ky_k\}$, we use $\mathcal{R}$ to denote the set $\{x_1,\ldots, x_k, y_1,\ldots, y_k\}$. The plane cubic graph   in Theorem \ref{basis} is denoted by $G^\prime$, and we use $C'$ to denote the boundary cycle of $G^\prime$,  which is   subdivided by vertices in $\mathcal{R}$ to obtain the boundary cycle $C$ of the cross cap in the embedding of $G$. For a vertex subset $X\subseteq V(G)$, let $X^\prime=X\cap V(G^\prime)$. For $e\in E(C')$,   the \textit{weight} $\omega(e)$ of $e$ is the number of vertices in $\mathcal{R}$ contained in $e$, i.e., $e$ is subdivided into a path of length $\omega(e)+1$ in $G$. For a subgraph $H$ of $G^\prime$, the weight of $H$ is defined as $\omega(H)=\sum\limits_{e\in E(H)\cap C'}\omega(e)$.

For a face $F$ of $G^\prime$, we use $C_F$ to denote the cycle bounding it. A face $F$ of $G^\prime$ is called a \textit{boundary face} if $C_F$ shares an edge with $C'$, and is called an \textit{internal face} otherwise. A boundary face $F$ is called a \textit{bridge face} if the subgraph induced by $C_F \cap C'$ is disconnected. Since $G^\prime$ is cubic, every connected component of $C_F \cap C'$ is a single edge of $C'$ possibly  subdivided by some vertices in $\mathcal{R}$.  

For an edge $e\in E(G)$, when we say $e$ is contained in an equilibrated cut $\partial_G(X)$, we assume the cut is minimized. This implies that both $G[X]$ and $G[X^c]$ are connected.

 \begin{observation}\label{obs-1}
     If $\partial(X)$ is an equilibrated cut  of $(G, \sigma)$ and $G[X]$ is connected, then $G^\prime[X^\prime]$ is also connected.
 \end{observation}
 \begin{proof}
 
 Assume to the contrary that $G^\prime[X^\prime]$ consists of $s(s \ge 2)$ connected components $ G^\prime[X^\prime_1], \cdots, G^\prime[X^\prime_s]$. These components are joint together by negative edges, hence $ d^-(X) < \sum\limits_{i=1}^s d^-(X_i)$. On the other hand, $\sum\limits_{i=1}^s d^+(X_i) = d^+(X)$ as there is no positive edges between different parts, 
  and $d^+(X)=d^-(X)$ as $\partial(X)$ is an equilibrated cut. So $\sum\limits_{i=1}^s d^+(X_i) < \sum\limits_{i=1}^s d^-(X_i)$ and there exists some $i$ such that $d^+(X_i)<d^-(X_i)$, contrary to the premise that the signature $\sigma$ is minimum.
 \end{proof}
 
By Observation \ref{obs-1}, in the drawing of a canonical embedding of $G$, an equilibrated cut can be embodied with a solid curve $\gamma$ in color red or blue on the plane that  traverses  a sequence  of positive edges and faces, with exactly two edges on $C'$, and splits $G'$ into two parts.  We denote a cut by this sequence consisting of faces and positive edges.
 For example,  $B=[e_1F_1e_2\ldots F_{r-1}e_r]$ represents an equilibrated cut formed by $r$ positive edges,  where $e_1, e_2 \in E(C' )$ and for $i=1,\ldots,r-1$, $e_i$ and $e_{i+1}$ are on $C_{F_i}$. 
 We say an equilibrated cut contains face $F$ (or edge $e$) if $F$ (or $e$) is contained in the sequence representing the cut. For two elements $x,y$ in the sequence of $B$, we denote by $B[x,y]$ the subsequence from $x$ to $y$. For example, for the cut $B=[e_1F_1e_2\ldots F_{r-1}e_r]$, $ B[e_1,e_3]=[e_1F_1e_2F_2e_3]$, $B[F_1, e_3] = [F_1e_2F_2e_3]$ and $B[F_1,F_2]=[F_1e_2F_2]$. 

\begin{observation}\label{matching}
For any equilibrated cut $\partial_G(X)$, the positive edges in $\partial_G(X)$ form a matching.   
\end{observation}
\begin{proof}
Assume to the contrary that an equilibrated cut $\partial_G(X)$ contains two edges incident to a vertex $w$. After switching at the cut and then switching at $w$, the resulting signature contains at most $k-1$ negative edges (since $d_G(w)=3$), contrary to the premise that $\sigma$ is minimum.
\end{proof}

The following lemma was proved in \cite{25}. 
\begin{lemma}\label{weight}
  If $k\geq 3$, $(G,\sigma) \in \mathcal{S}^*(k)$ and $E^-_\sigma(G)=\{x_1y_1,\ldots, x_ky_k\}$, then the following claims hold:

\begin{enumerate}
    \item $\omega(F)\leq 2$ for any boundary face $F$, and  $\omega(F)=0$ if  $F$ is a bridge face. 
    \item  If $F_1$ and $F_2$ are boundary faces whose edges on $C'$ are consecutive, then $\omega(F_1)+\omega(F_2)\leq 3$.
    \item  When $F$ is a bridge face, for any connected component $A_F$ of $G'\setminus(E(C_F)\cap E(C'))$, $\omega(A_F)\in \{2, 2k-2\}$.
\end{enumerate} 
\end{lemma}

\begin{proposition}\label{bridgeface}
Assume $k\geq 4$ and $F$ is a bridge face. Then $G^\prime\setminus(E(C_F)\cap E(C'))$ has exactly two components $A_1$ and $A_2$, such that  $A_1$ is a single face bounded by a copy of $C_4$ in $G$ and $\omega(A_2)=2k-2$.
Moreover, the length of $C_F$ is at least 6.
\end{proposition}
\begin{proof}
Assume $A_1, A_2,\ldots, A_s$ are the connected components of $G^\prime\setminus(E(C_F)\cap E(C'))$, in this cyclic order. By Lemma \ref{weight}, each connected component has weight $2$ or $2k-2$. If $s > 2$,  then as the total weight is $2k$, we have $s=k$, and each $\omega(A_i)=2$ for every $i\in [s]$. Then the cut $\partial_G(A_1 \cup A_2)$ is a $(2,4)$-cut in $(G,\sigma)$, a contradiction (see Figure \ref{bridge faces}). Thus $s=2$, and  we may assume that $\omega(A_1)=2$ and $\omega(A_2)=2k-2$.

\begin{figure}[H]
\centering
\subfigure{\includegraphics[width=9cm]{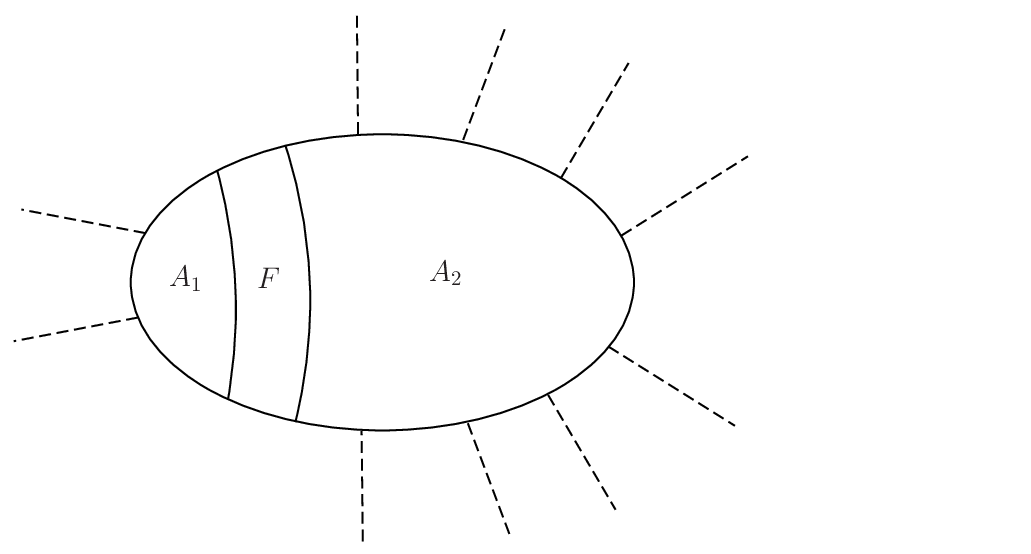}}
\subfigure{\includegraphics[width=8cm]{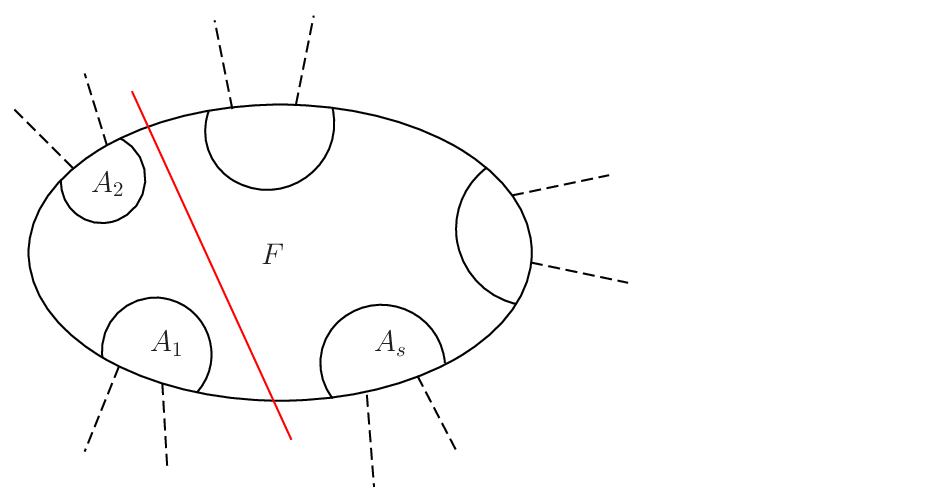}}
\caption{Two types of bridge faces.}
\label{bridge faces}
\end{figure}

Let $E(C_F)\cap E(C') = \{e_1, e_2\}$, and $V(A_1)\cap \mathcal{R} =\{x_1, x_2\}$. Let $e_1^\prime$ and $e_2^\prime$ be the two unweighted edges in $E(G)$ of $E(A_1)\cap E(C)$ adjacent to $e_1$ and $e_2$ respectively. Let $e_3$ be an arbitrary edge on $C$ between $x_1$ and $x_2$.

\begin{claim}\label{clm-0}
If $B$ is an equilibrated cut containing $e_3$, then $B$ contains $F$ as well as an edge in $E(A_2)\cap C’$; if $B$ is an equilibrated cut containing $e'_1$ (or $e'_2$), then $B$ does not contain $F$. 
\end{claim}
\begin{proof}
 Assume $B$ is an equilibrated cut containing $e_3$. Let $Y_1,Y_2$ be the two parts separated by $B$. If $B$ does not contain $F$, then one of $Y_1$ and $Y_2$, say $Y_1$, is contained in $V(A_1)$. $B$ does contain negative edges, so $x_1$ and $x_2$ are separated by $B$. As a result, $B=\partial (Y_1)$ consists of one negative edge and at least two positive edges, contrary to the premise that $B$ is equilibrated. So $B$ contains $F$, and hence contains an edge in $E(A_2)\cap C’$. 

 Assume $B$ is an equilibrated cut containing $e'_1$. If $B$ contains $F$, then we replace the subsequence $B[F,e_1^\prime]$ with $[Fe_1]$, see Figure \ref{22}, where $B$ is represented by a solid red curve $\gamma$. To show the replacement, we let $\gamma$ veer to $e_1$ when it traverses $F$. The resulting sequence is a cut $B'$ with less positive edges, but with the same number of negative edges, a contradiction. Thus $B$ does not contain $F$. The case $B$ contains $e'_2$ is symmetric.
\end{proof}

\begin{figure}[H]
\centering
\includegraphics[width=10cm]{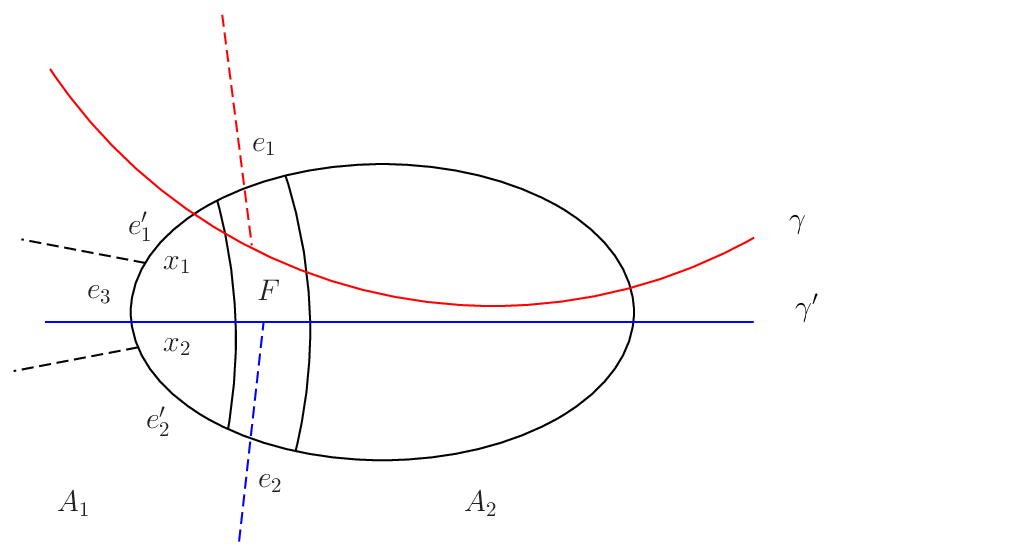}
\caption{}
\label{22}
\end{figure}

It follows from Claim \ref{clm-0} that  for $i \in \{1,2\}$, $e_i^\prime$ can only be contained in a $(2,2)$-cut including $x_1y_1$, $x_2y_2$ as well as an edge in $E(A_1)\cap C'$. By planarity, the only possible case is that there is a face $F^\prime$ whose boundary cycle $C_{F^\prime}$ contains both $e_1^\prime$ and $e_2^\prime$.

The edge $e_3$ must be on $C_{F^\prime}$ too, for otherwise, take an arbitrary equilibrated cut $D$ containing $e_3$, where we use the solid blue curve $\gamma'$ to represent it in Figure \ref{22}. By Claim \ref{clm-0}, $D$ contains $F$, and we replace the subsequence $D[F, e_3]$ with $[Fe_2]$. Similarly, the replacement is shown in Figure \ref{22} by letting the solid blue curve $\gamma'$ representing $D$ veer to $e_2$ when it traverses $F$. Compared to $D$, the resulting cut $D'$ loses at least two positive edges and at most one negative edge, a contradiction too.
By the arbitrariness of $e_3$, all edges between $x_1$ and $x_2$ are on $C_{F^\prime}$. From Lemma \ref{weight}, $F^\prime$ cannot be a bridge face, hence $e_3=x_1x_2$ and $A_1$ in $G$ is a copy of $C_4$ which is the boundary of  $F^\prime$.

If the length of $C_F$ is less than 6,  then $|E(C_F)\cap E(A_2)|\leq 2$. Let $K$ be an equilibrated cut containing $e_3$, which also contains $F$. Then $K$ contains an edge of $E(C_F)\cap E(A_2)$ which is adjacent to $e_1$ or $e_2$, say $e_1$, and we replace $K[F, e_3]$ with $[Fe_1]$. The resulting cut $K'$ loses one positive edge and at most one negative edge, so $K'$ is also an equilibrated cut. However, the positive edges of $K'$ do not form a matching, contrary to Observation \ref{matching}. Thus $C_F$ has length at least $6$. This completes the proof of {Proposition} \ref{bridgeface}.
\end{proof}

Since the two positive edges in a $(2,2)$-equilibrated cut lie on a single boundary face, we have the following corollary.
\begin{corollary}\label{2+2}
Every $(2,2)$-cut has the form $\partial_G(X)$, where $G[X]\cong K_2$ or $G[X]$ is a single face bounded by a $4$-cycle.
\end{corollary}

The following proposition describes the structure of cuts that contain exactly three positive edges.
\begin{proposition}\label{3+1}
Let $\partial_G(X)$ be a cut containing exactly three positive edges, with $\omega(G[X])\leq \omega(G[X^c])$. 
\begin{enumerate}
\item If $\partial_G(X)$ is a $(3,1)$-cut, then $G[X]\cong K_2 $ or $G[X] \cong C_4$.

\item If $\partial_G(X)$ is a $(3,2)$-cut, then $G[X]$ contains at most one cycle.
\end{enumerate}
\end{proposition}
\begin{proof}

Denote the sequence of $\partial_G(X)$ by $B=[e_1F_1eF_2e_2]$, where $e_1$ and $e_2$ are weighted edges on $C’$, and $F_1$ and $F_2$ are two adjacent boundary faces whose common edge is $e=w_1w_2$, with $w_1\in X$ and $w_2\in X^c$. As $\omega(G[X])\leq \omega(G[X^c])$, $\omega(G[X])=1$ or $2$.

Assume first that $e_1$ and $e_2$ are adjacent on $C'$. Let $w$ be the common end vertex of $e_1$ and $e_2$. Then $w$ is incident to both $F_1$ and $F_2$, and $X - \mathcal{R}= \{w\}$. Hence $G[X]$ is a path consisting of $w$ and one or two vertices of $\mathcal{R}$. 

Assume next that $e_1$ and $e_2$ are not adjacent on $C'$. Let $e'_i \in E(G[X])$ be the unweighted edge of $C$ adjacent to $e_i$ for $i=1,2$.

\begin{figure}[H]
\centering
\includegraphics[width=11cm]{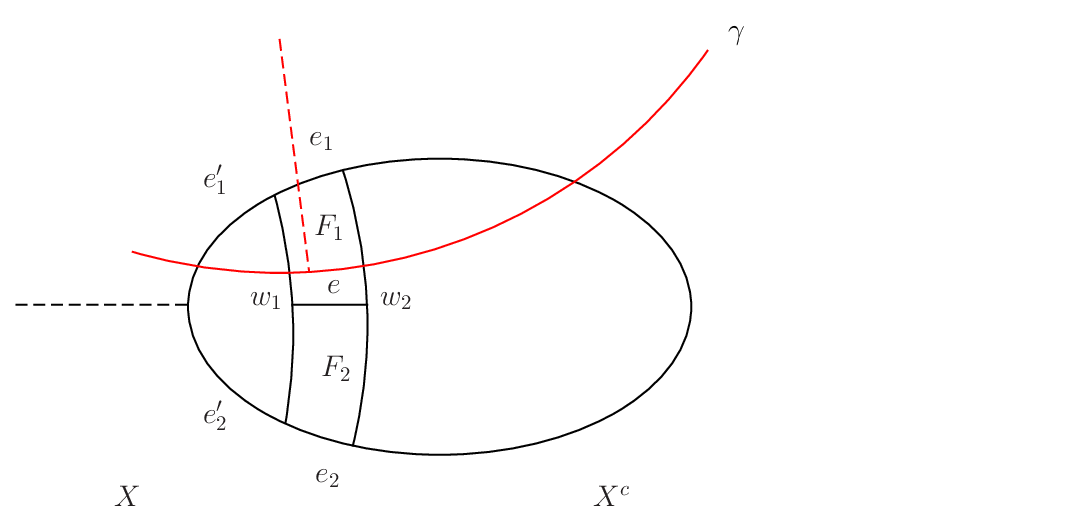}
\caption{}
\label{31}
\end{figure}

If $\partial_G(X)$ is a $(3,1)$-cut, then $\omega(G[X])=1$. Since $X$ contains only one vertex in $\mathcal{R}$, each equilibrated cut $D$ including $e_1^\prime$ contains $F_j$ for some $j\in \{1,2\}$. When $j=1$, we replace $D[F_1,e'_1]$ with $[F_1e_1]$, see Figure \ref{31}, where $D$ is shown with the solid red curve $\gamma$. Compared to $D$, the resulting cut $D'$ loses at least one positive edge but no negative edge, a contradiction. So $D$ contains $e'_1$ as well as $F_2$. 
It is implied that there exists a face $F_1^\prime$ adjacent to $F_2$ with $e_1^\prime\in C_{F_1^\prime}$, for otherwise we can replace $D[F_2,e'_1]$ with $[F_2e_1]$ and obtain a contradiction again.

Symmetrically, there exists a face $F_2^\prime$ adjacent to $F_1$ with $e_2^\prime\in C_{F_2^\prime}$. By planarity of $G^\prime$, $F^\prime_{1}=F^\prime_{2}$, and moreover, $F^\prime_{1}$ cannot be a bridge face by our descriptions in Lemma \ref{weight}. So $G[X]$ is exactly a single face bounded by a copy of $C_4$.  

If $\partial_G(X)$ is a $(3,2)$-cut, then $\omega(G[X])=2$. Let $X\cap\mathcal{R}= \{x_1, x_2\}$, and furthermore we may presume $x_1$ and $x_2$ are not on $e_1$ or $e_2$, for otherwise $G[X]$ is formed by a subgraph generating a $(3,1)$-cut with a vertex of degree one attached to it, and thus contains at most one cycle.
 
We arbitrarily select an edge on $C$ between $x_1$ and $x_2$ and denote it by $e_3$. 
For $i=1,2$, when an equilibrated cut including $e_i^\prime$ contains neither $F_1$ nor $F_2$, it can only be a $(2,2)$-cut with form $\partial_G(Y)$, where $Y\subseteq V(A_1)$. The two positive edges in $\partial_G(Y)$ are $e_i^\prime$ as well as another edge lying between $x_{3-i}$ and $e_{3-i}$, and there is a face connecting them. On the other hand, if an equilibrated cut containing $e_i^\prime$ does include $F_1$ or $F_2$ in its sequence, then it must be $F_{3-i}$, and there is a face linking $e_i$ to it. The reason is the same as the case $\partial_G(X)$ is a $(3,1)$-cut, and we omit the details here.

Now focus on $e_3$. Let $K$ be an equilibrated cut containing $e_3$. $K$ must include $F_1$ or $F_2$, and we let it be $F_1$ without loss of generality. Then there exists a face adjacent to $F_1$ whose boundary cycle contains $e_3$, for otherwise we can replace $K[F_1,e_3]$ with $[F_1e_1]$ and obtain a new sequence $K'$. $K'$ is a cut with a loss of at least two positive edges and at most one negative edge compared to the equilibrated cut $K$, which is a contradiction.

\begin{figure}[H]
\centering
\includegraphics[width=11cm]{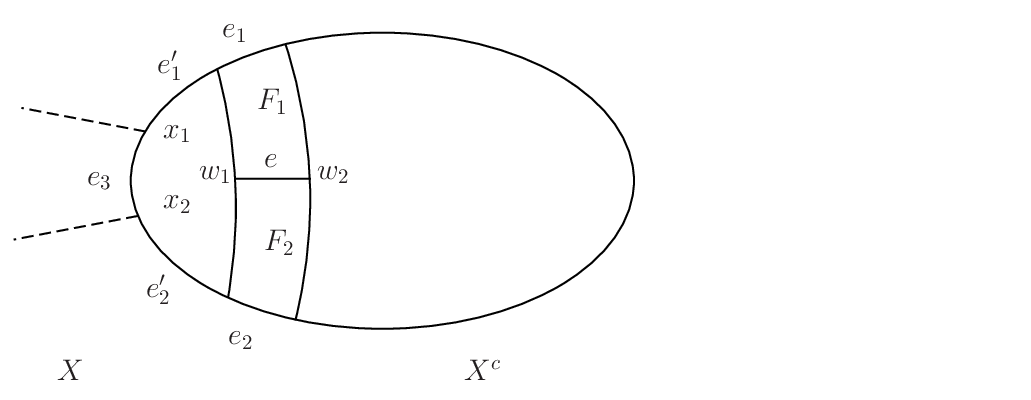}
\caption{}
\label{32}
\end{figure}

In the end, using $C_X$ to represent the cycle pieced together by $C$, $F_1$ and $F_2$ in $G[X]$, we make a summary of the appeals of $e_1^\prime$, $e_2^\prime$ and $e_3$. $e_1'$ demands a face $F^\prime_{1}$ linking $e'_1$ with some edge on $C_X$ between $x_2$ and $w_1$; $e_2'$ demands a face $F^\prime_{2}$ linking $e'_2$ with some edge on $C_X$ between $x_1$ and $w_1$; $e_3$ demands a face $F^\prime_{3}$ linking $e_3$ with an edge lying on the boundary of $F_1$ or $F_2$. By planarity, the unique solution to meet all of them is that $F^\prime_{1}=F^\prime_{2}=F^\prime_{3}$. By Lemma \ref{weight}, $F^\prime_{1}$ cannot be a bridge face, so $X$ is exactly a single face bounded by a $C_5$.
\end{proof}

We now have characterized the embodiment of some special small cuts in a certain canonical embedding of $(G, \sigma)$. Surprisingly, one can obtain a general property of the underlying graph $G$, which does not depend on how it is embedded on the projective plane.
\begin{corollary}\label{4cut}
Let $\partial_G(X)$ be a cut consisting of four edges, then $G[X]$ or $G[X^c]$ contains at most one cycle.
\end{corollary}
\begin{proof}
First assume that one of $G[X]$ and $G[X^c]$ has a bridge, say $G[X]$. Since $G$ is essentially $4$-edge-connected, it is not difficult to deduce that $G[X]\cong K_2$. So from now on both $G[X]$ and $G[X^c]$ are supposed to be $2$-edge-connected. We arbitrarily select a minimum signature $\sigma$ as well as a canonical embedding of $(G,\sigma)$, and let $G[X]$ be the part containing less number of faces. Since we have already solved the cases when $\partial_G(X)$ is a $(2,2)$ or $(3,1)$-cut in Corollary \ref{2+2} and Proposition \ref{3+1}, all the four edges of $\partial_G(X)$ are deemed to be positive under this embedding. Arbitrarily select one of them an label it with $e_1$. $e_1$ is contained in an equilibrated cut $\partial_G(Y)$, and we use the solid red curve $\gamma$ to show $\partial_G(Y)$ in Figure \ref{planecut}.

We now adopt a planar view of $(G,\sigma)$ based on $G^\prime$, in which every negative edge $x_iy_i$ in $E^-_\sigma(G)=\{x_1y_1,\ldots, x_ky_k\}$ is regarded to be broken by the cross cap into two halves on the plane, one linked to $x_i$ and one to $y_i$. We switch at $\partial_G(Y)$ and obtain another minimum signature $\mu$. On the plane, this operation is shown visually by the following procedure:
\begin{enumerate} 

\item Cut along $\gamma$ to separate $G$ into two pieces induced by $G[Y]$ and $G[Y^c]$. Every positive edge in $\partial_G(Y)$ is regarded to be broken into two halves by the cut.

\item Invert the piece $G[Y]$ and move it to another side of $G[Y^c]$.

\item Match every pair of halves belonging to $\partial_G(Y)\cap E_\sigma^-(G)$ into a complete edge.

\item Change the sign of every edge in $\partial_G(Y)$.

\end{enumerate}

\begin{figure}[H]
\centering
\includegraphics[width=12.7cm]{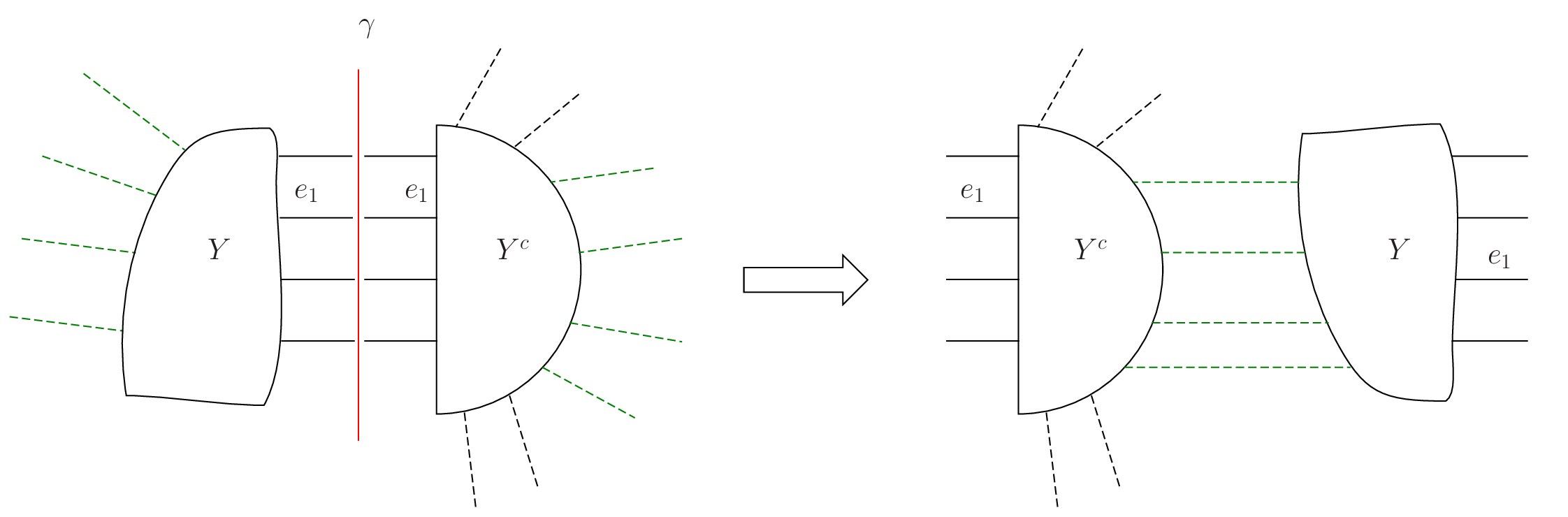}
\caption{}
\label{planecut}
\end{figure}

The process is displayed in Figure \ref{planecut}, where we use color green to highlight the negative edges belonging to $\partial_G(Y)$. We assert that the region corresponding to $G[X]$ is not separated in the planar view, i.e. $\partial_G(Y)\cap E(G[X])=\emptyset$. If not, assume $X$ is partitioned into at least two subsets $X=\bigcup\limits_{j=1}^{t} X_j$ ($t\geq 2$), then for each $j$, $|\partial_G(Y)\cap E(G[X_j])|\geq 2$ because $G[X]$ is $2$-edge-connected. On the other hand, there are at most three edges in $\partial_G(X)\cap E_\mu^+(G)$ after the switching ($e_1$ becomes negative), which implies that there is some component $X_j$ satisfying $|\partial_G(X_j)\cap E_\mu^+(G)|\leq 1$ and $|\partial_G(X_j)\cap E_\mu^-(G)|\geq 2$ at the same time, which is a contradiction.

Remember that the positive edges in any equilibrated cut form a matching, we finally obtain a new canonical embedding as a result after the switching, under which $\partial_G(X)$ becomes a $(2,2)$ or $(3,1)$-cut. From the perspective of the whole projective plane, the embedding of the underlying graph $G$ keeps unchanged, so the total number of faces is a constant, and thus $G[X]$ is still the part containing less number of faces. From the essence of Corollary \ref{2+2} and Proposition \ref{3+1}, $G[X]$ is exactly a single face bounded by a copy of $C_4$.
\end{proof}

\subsection{Finiteness of boundary faces}
In this part, we will exploit the properties of small cuts we have proposed in the last subsection to restrict the cardinality of boundary faces in canonical embeddings. For the convenience of analysis, we first standardize the canonical embedding of a signed graph in $\mathcal{S}^*(k)$. We say a canonical embedding is \textit{normal} if it contains no bridge face.
\begin{proposition}
For $k\geq 4$, every signed graph $(G,\sigma)\in S^*(k)$ admits a normal embedding up to switching equivalence.
\end{proposition}
\begin{proof}
For $(G,\sigma)\in S^*(k)$, we arbitrarily fix a canonical embedding of it. Let $F$ be a bridge face if it truly exists, and denote the $C_4$-face adjacent to $F$ by $F^\prime$ according to Proposition \ref{bridgeface}. Let the two vertices of $\mathcal{R}$ on $C_{F'}$ be $x_1$ and $x_2$ without loss of generality. We switch at the equilibrated cut $\partial_G(\{x_1,x_2\})$, and as explained in the proof of Corollary \ref{4cut}, it is equivalent to clipping off the $K_2$ formed by $\{x_1,x_2\}$ on the plane and then move it to the other side inversely.

We assert that in the resulting canonical embedding, the number of bridge faces has been strictly decreased by the switching. If not, then the only possible case is that the antipodals of $x_1$ and $x_2$: $y_1$ and $y_2$ are on the boundary cycle of another boundary face, and furthermore, $\{x_1,x_2,y_1,y_2\}$ induces a $C_4$ enclosing a face embedded in the cross cap, denoted by $F^{\prime\prime}$. Consequently, the two adjacent $C_4$-faces $F^\prime$ and $F^{\prime\prime}$ generate a cut consisting of four edges in $G$, which is a contradiction to Corollary \ref{4cut}.

So starting from the initial canonical embedding, by sequentially switching at the $(2,2)$-cuts formed by $K_2$, we are finally able to eliminate all bridge faces and obtain a normal embedding.
\end{proof}

Recall that a \textit{composition} of an integer $N$ is a way of writing $N$ as the sum of an ordered sequence of positive integers. \textit{Cyclic compositions} of an integer $N$ is defined as equivalent classes on the set of all compositions of $N$ such that, two compositions belong to the same class if and only if one can be obtained from the other by a cyclic shift \cite{cc}. By Lemma \ref{weight}, the weight distribution of boundary faces in every normal embedding of $(G,\sigma)\in \mathcal{S}^*(k)$ determines a cyclic composition of $2k$ consisting of $1$ and $2$, possibly interlaced by segments of consecutive $0$s. We demonstrate the finiteness of boundary faces for $k=4$ and $5$ by limiting the length of each $0$-segment.

\begin{proposition}\label{finite}
Arbitrarily fix a normal embedding of $(G,\sigma)\in \mathcal{S}^*(k)$.
\begin{enumerate}
\item When $k=4$, there do not exist two consecutive edges on $C'$ of weight $0$;

\item When $k=5$, the length of consecutive $0$s will not exceed $3$.
\end{enumerate}
\end{proposition}
\begin{proof}
Fixed a normal embedding of $(G,\sigma)$. Assume $e_1$ and $e_2$ are two consecutive edges on $C'$ with $\omega(e_1)=\omega(e_2)=0$, and let $F_1$, $F_2$ be the two boundary faces they belong to respectively. We use $e$ to denote the edge $E(C_{F_1})\cap E(C_{F_2})$. Every equilibrated cut containing $e$ has at least four positive edges, and if the number of positive edges is exactly four, then one of $e_1$ and $e_2$ must be included in the cut, contrary to Observation \ref{matching}. So edges of weight $0$ will not be consecutive on $C'$ when $(G,\sigma)\in \mathcal{S}^*(4)$.

When $(G,\sigma)\in \mathcal{S}^*(5)$, every equilibrated cut containing $e$ has exactly five positive edges. To achieve this, there exist two more distinct boundary faces $F_1^\prime$ and $F_2^\prime$,
with $F_i^\prime$ adjacent to $F_i$ ($i=1,2$), see Figure \ref{5cut}, where we use the red curve $\gamma$ to show the cut. We use notation $B$ to denote it.

\begin{figure}[H]
\centering
\includegraphics[width=7cm]{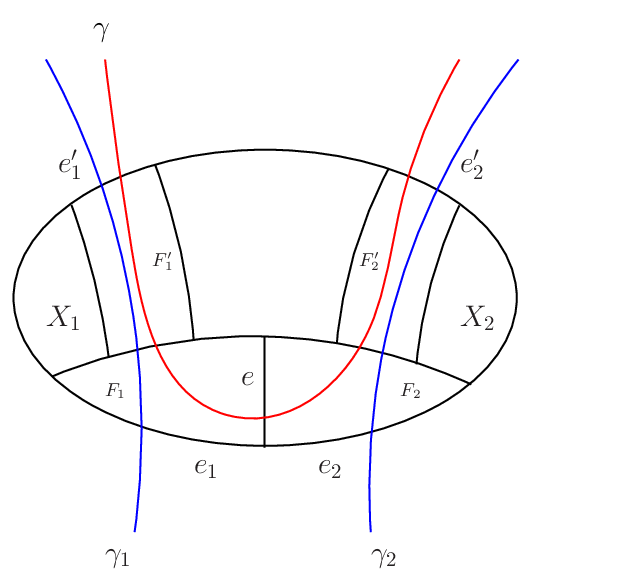}
\caption{}
\label{5cut}
\end{figure}

Let $e_1^\prime=C_{F_1^\prime}\cap C'$ and $e_2^\prime=C_{F_2^\prime}\cap C'$. $B=[e_1^\prime F_1^\prime \ldots F_2^\prime e_2^\prime]$ can be split into two sequences $B_1=[e_1^\prime F_1^\prime (E(C_{F_1^\prime})\cap E(C_{F_1}))F_1 e_1]$ and $B_2=[e_2^\prime F_1^\prime (E(C_{F_2^\prime})\cap E(C_{F_2}))F_2 e_2]$, each of which represents a cut containing exactly three positive edges. In Figure \ref{5cut}, they are shown by another two blue curves $\gamma_1$ and $\gamma_2$. The two cuts can be written in the form $\partial_G(X_1)$ and $\partial_G(X_2)$ respectively, where $X_i\subseteq V(G)$ is the peripheral part separated by $\gamma_i$ in Figure \ref{5cut} not containing ends of $e$ ($i=1,2$). By assumption, $d^-(X_1)+d^-(X_2)=5$, and without loss of generality, we let $\partial_G(X_1)$ be the cut with less negative edges. It can be deduced from Proposition \ref{3+1} that either $G'[X_1]$ is a single boundary face of positive weight lying between $F_1$ and $F_1^\prime$, or $F_1$ and $F_1^\prime$ are two consecutive boundary faces on $C'$. Anyway, when $k=5$, one of $e_1$ and $e_2$ must meet a positively weighted edge on $C'$, which implies the length of consecutive $0$s will not exceed $3$.
\end{proof}

\subsection{Transmission of finiteness}

For $(G,\sigma)\in \mathcal{S}^*(k)$, let $\mathcal{F}$ be the face set of $(G,\sigma)$ under a normal embedding. When $k=4$ or $5$, $\mathcal{F}$ admits a natural partition $\mathcal{F}=\mathcal{F}_0\cup \mathcal{F}_1 \cup \mathcal{F}_2$, where $\mathcal{F}_0$, $\mathcal{F}_1$ and $\mathcal{F}_2$ are the sets of faces in the cross cap, boundary faces and internal faces respectively. $\mathcal{F}_0$ has cardinality $k$, and we have already shown that $|\mathcal{F}_1|$ is bounded in the above section. In this part, we are going to build a map from $\mathcal{F}_2$ to $\mathcal{F}_1 \times \mathcal{F}_1$ to transmit the message of finiteness from $\mathcal{F}_1$ to $\mathcal{F}_2$.

\begin{lemma}\label{inner3}
Fix a canonical embedding of $(G,\sigma)\in \mathcal{S}^*(k)$ $(k\geq 4)$. Not considering the outer face corresponding to the cross cap, for any two fixed faces of $G^\prime$, there are at most three faces of $G'$ adjacent to both of them. Especially, for two boundary faces whose edges on $C'$ are consecutive, they have at most one common adjacent face of $G'$.
\end{lemma}
\begin{proof}
\begin{figure}[H]
\centering
\includegraphics[width=7cm]{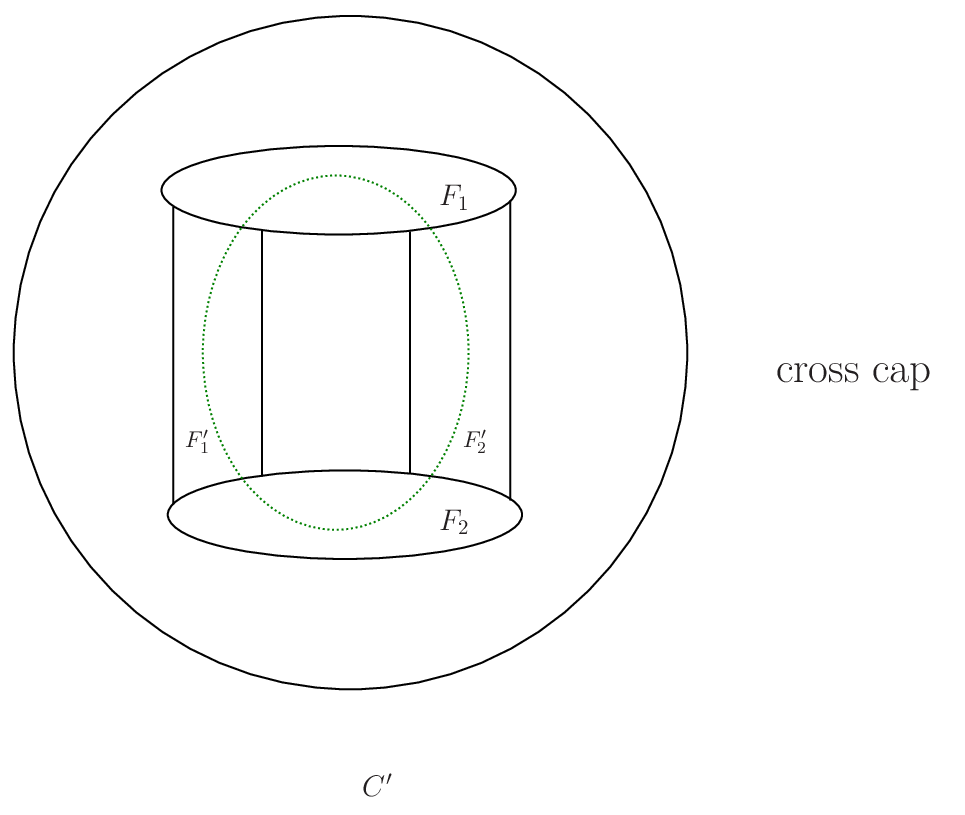}
\caption{}
\label{common}
\end{figure}
We use $F_1$ and $F_2$ to represent the two faces of $G'$ in question who have at least two common adjacent faces. The common neighboring faces of $F_1$ and $F_2$ are arranged in cyclic order. Among them there is one corresponding to the cross cap of projective plane, and we take the two faces of $G^\prime$ adjacent to it, say $F_1^\prime$ and $F_2^\prime$. The region enclosed by the four faces $F_1$, $F_2$, $F_1^\prime$ and $F_2^\prime$ generates a cut of at most four edges, see Figure \ref{common}, where the cut is marked with a dotted green cycle. From the rules of weight distribution of boundary faces in Lemma \ref{weight}, it can be deduced that when $k\geq 4$, the number of boundary faces is at least $6$, which implies that there exist another two boundary faces $F_3,F_4 \notin \{F_1, F_2, F_1^\prime, F_2^\prime\}$ lying outside the region. So by Corollary \ref{4cut}, the region is a single face bounded by a copy of $C_4$ when $F_1^\prime$ and $F_2^\prime$ are not adjacent, or a copy of $K_2$ otherwise. Anyway, the total number of faces adjacent to both $F_1$ and $F_2$ will not exceed 3.

In particular, when $F_1$ and $F_2$ are two adjacent boundary faces whose edges on $C'$ are consecutive, let $xy$ be the common edge of $E(C_{F_1})\cap E(C_{F_2})$ with $y\in V(C')$. Not considering the outer face corresponding to the cross cap, a common adjacent face of $F_1$ and $F_2$ must contain vertex $x$, for otherwise the three faces will generate an essential cut of size $3$ in $G$, contrary to the fact that $G$ is essentially 4-edge-connected. It is deduced that the common adjacent face of $F_1$ and $F_2$ is unique.
\end{proof}

\begin{lemma}\label{inner2}
Every internal face is adjacent to at least two boundary faces.
\end{lemma}
\begin{proof}
The number of positive edges in an equilibrated cut containing the common edge of two internal faces is at least five, so the claim is obvious when $k=4$. We only consider the case $k=5$. First notice that every internal face is adjacent to at least one boundary face. On the contrary, suppose there is an internal face $I$ whose neighboring faces are all internal faces too, then $\forall e'\in E(C_I)$, the equilibrated cut containing $e'$ must include at least six positive edges, which is impossible for a signed graph in $\mathcal{S}^*(5)$.

Arbitrarily take an internal face $I_1$, and let $F$ be a boundary face adjacent to it. Since $G$ is cubic, $I_1$ and $F$ share at least two common neighboring faces, from which we randomly select one and label it with $I_2$. Let $e=E(C_{I_1})\cap E(C_{I_2})$. If $I_2$ is also an internal face, then every equilibrated cut containing $e$ has exactly five positive edges, and moreover, $F$ will not appear in the sequence of the cut, for otherwise one of $E(C_{I_1})\cap E(C_F)$ and $E(C_{I_2})\cap E(C_F)$ will be in the sequence too, which is a contradiction to Observation \ref{matching}. So we conclude that there exist two distinct boundary faces $F_1$ and $F_2$ different from $F$, such that $F_i$ is adjacent to $I_i$ ($i=1,2$). Thus the lemma is proved.
\end{proof}

Based on Lemma \ref{inner3} and Lemma \ref{inner2}, one can naturally construct a map $\eta: \mathcal{F}_2 \mapsto \mathcal{F}_1 \times \mathcal{F}_1$ with the property that, for any fixed pair $(F_1,F_2)\in \mathcal{F}_1 \times \mathcal{F}_1$, $|\eta^{-1}((F_1,F_2))|\leq 3$. We are now able to transmit the message of finiteness via map $\eta$. For a signed graph $(G,\sigma)\in \mathcal{S}^*(4)$ or $\mathcal{S}^*(5)$, $G$ admits a normal embedding $\Sigma$ with at most $M$ boundary faces, where $M$ is an undetermined constant. From the embedding we construct an auxiliary multigraph $Q$ in the following way:

Let $V(Q)=\mathcal{F}_1$, and arrange all vertices on a cycle $C_M$ according to the order of boundary faces in $\Sigma$. Embed $C_M$ on the plane and let it enclose a region $\Omega$. For every internal face $I$ in $G^\prime$, a pair $\eta(I)\in \mathcal{F}_1 \times \mathcal{F}_1$ is output, then correspondingly we add the edge representing $\eta(I)$ within $\Omega$.

Notice that $Q$ may not be unique due to the selection of $\eta(I)$ at every step. However, it is practicable to show that all possible $Q$ admit a uniform upper bound on the cardinality of edges. By Lemma \ref{inner3}, the multiplicity of every multiedge of $Q$ is at most $3$, and especially, all edges on $C_M$ are simple. Since $Q$ is a planar graph, we apply Euler's Formula to $Q$ . As a result, it is shown that the cardinality of $|E(Q)|$ is no more than $4M-9$.

In the end, it remains to determine a certain value for the upper bound $M$. Remember that every normal embedding of $(G,\sigma)\in \mathcal{S}^*(k)$ corresponds to a cyclic composition of $2k$ consisting of $1$, $2$ interlaced by some $0$s. In \cite{cc}, the generating function of cyclic composition number was derived by Hadjicostas.

\begin{lemma}
Let $A$ be a subset of $\mathbb{Z}^+$, and $c_A(n)$ be the number of cyclic compositions of $n$ such that every integer in the sum is contained in $A$. The generating function of $c_A(n)$ is
\begin{displaymath}
\sum\limits_{n\geq 1} c_A(n) x^n=\sum\limits_{n\geq 1} \frac{\phi(n)}{n} \log\frac{1}{1-\sum\limits_{s\in A}x^{sn}},
\end{displaymath}
where $\phi(n)$ is Euler's totient function at $n$.
\end{lemma}

Set $A=\{1,2\}$, and apply Taylor series $\log(1+x)=\sum\limits_{i\geq 1} \frac{(-1)^{i-1}}{i} x^i$ to the generating function, then we get the equation

\begin{displaymath}
\sum\limits_{n\geq 1} c_A(n) x^n=\sum\limits_{n\geq 1}\sum\limits_{i\geq 1} \frac{\phi(n)}{n}\cdot\frac{x^{ni}(1+x^n)^i}{i}.
\end{displaymath}

For $k\in \mathbb{Z}^+$, comparing the coefficients of $x^k$ on both sides, we have
\begin{displaymath}
c_A(k)=\sum\limits_{d|k}\sum\limits_{j=0} \frac{\phi(d)\binom{\frac{k}{d}-j}{j}}{k-dj}.
\end{displaymath}
Taking $0$s into consideration, by Proposition \ref{finite}, we set $M=3\cdot 2^4=48$ when $k=4$ and $M=3\cdot 4^5$ when $k=5$.

By the construction of $Q$, $|E(Q)|$ is exactly the value $|\mathcal{F}_1|+|\mathcal{F}_2|$ in the normal embedding of $(G, \sigma)$. We finally prove the finiteness of $\mathcal{S}^*(4)$ and $\mathcal{S}^*(5)$ by showing that the face number of the members is limited.

\section{Acknowledgement}

The author would like to express sincere gratitude to Professor Reza Naserasr and Xuding Zhu, for their constructive suggestions in the refinement of the script. The author would also like to thank Deping Song, Lujia Wang and Zhouningxin Wang for their helpful comments during the writing of the paper.

\end{document}